\documentclass[11pt,a4paper]{article}
\usepackage{amsthm, amssymb, amsmath, a4wide}
\usepackage[all]{xy} 
\usepackage{paralist} 
\usepackage{graphicx} 
\usepackage[sans]{dsfont} 
\usepackage{varioref} 
\usepackage{subfigure}
\usepackage{tikz}
\usepackage{pgf}
\usetikzlibrary{decorations.pathreplacing,calc}
\usepackage{xifthen}
\usepackage{hyperref}

\title{How to determine the {\em sign} of a valuation on $\cc[x,y]$?}
\author{Pinaki Mondal\\ Weizmann Institute of Science\\ pinaki@math.toronto.edu}

\DeclareMathOperator{\nef}{Nef}
\DeclareMathOperator{\NE}{NE}

\makeatletter

\newcommand{\Rmnum}[1]{\expandafter\@slowromancap\romannumeral #1@}
\makeatother

\DeclareMathOperator\ord{ord}

\DeclareMathOperator\proj{Proj}

\DeclareMathOperator\spec{Spec}

\newcommand{\scrC}{\ensuremath{\mathcal{C}}}

\newcommand{\scrI}{\ensuremath{\mathcal{I}}}

\newcommand{\scrM}{\ensuremath{\mathcal{M}}}

\newcommand{\cc}{\ensuremath{\mathbb{C}}}

\newcommand{\pp}{\ensuremath{\mathbb{P}}}
\newcommand{\qq}{\ensuremath{\mathbb{Q}}}
\newcommand{\rr}{\ensuremath{\mathbb{R}}}
\newcommand{\zz}{\ensuremath{\mathbb{Z}}}

\newcommand{\dsum}{\ensuremath{\bigoplus}}

\newtheorem{thm}{Theorem}[section]
\newtheorem*{thm*}{Theorem}
\newtheorem{lemma}[thm]{Lemma}
\newtheorem*{lemma*}{Lemma}
\newtheorem{lemma-in-thm}{Lemma}[thm]

\newtheorem{prop}[thm]{Proposition}
\newtheorem*{prop*}{Proposition}
\newtheorem{cor}[thm]{Corollary}

\newtheorem*{claim*}{Claim}

\newtheorem*{conjecture*}{Conjecture}

\theoremstyle{definition}

\newtheorem*{constrinition*}{Construction-Definition}

\newtheorem*{convention*}{Convention}
\newtheorem{defn}[thm]{Definition}
\newtheorem*{defn*}{Definition}

\newtheorem*{definotation*}{Definition-Notation}
\newtheorem{example}[thm]{Example}
\newtheorem*{example*}{Example}

\newtheorem*{fact*}{Fact}

\newtheorem*{facts*}{Facts}

\newtheorem{notation}[thm]{Notation}

\newtheorem*{bold-note*}{Note}
\newtheorem{bold-question}[thm]{Question}
\newtheorem{rem}[thm]{Remark}

\newtheorem*{reminition*}{Remark-Definition}

\newtheorem{remexample*}{Remark-Example}

\newtheorem*{remtation*}{Remark-Notation}

\newtheorem*{remuestion*}{Remark-Question}

\theoremstyle{remark}
\newtheorem*{rem*}{Remark}
\newtheorem*{note*}{Note}
\newtheorem*{notation*}{Notation}
\newtheorem*{question*}{Question}

\newtheorem*{questions*}{Questions}

\newcounter{UnorderedProofTempCtr}
\newcommand{\tempcommand}{}

\newcommand{\lot}{\text{l.o.t.}}

\newcommand{\dpsx}[2]{{#1} \langle \langle #2 \rangle \rangle }
\newcommand{\dpsxc}{\dpsx{\cc}{x}}

\begin{document}
\maketitle

\begin{abstract} 
Given a divisorial discrete valuation {\em centered at infinity} on $\cc[x,y]$, we show that its sign on $\cc[x,y]$ (i.e.\ whether it is negative or non-positive on $\cc[x,y]\setminus \cc$) is completely determined by the sign of its value on the {\em last key form} (key forms being the avatar of {\em key polynomials} of valuations \cite{maclane-key} in `global coordinates'). The proof involves computations related to the cone of curves on certain compactifications of $\cc^2$ and gives a characterization of the divisorial valuations centered at infinity whose {\em skewness} can be interpreted in terms of the {\em slope} of an extremal ray of these cones, yielding a generalization of a result of \cite{favsson-eigen}. A by-product of these arguments is a characterization of valuations which `determine' normal compactifications of $\cc^2$ with one irreducible curve at infinity in terms of an associated `semigroup of values'.  

\end{abstract}

\section{Introduction} \label{sec-intro}
\begin{notation}
Throughout this section $k$ is a field and $R$ is a finitely generated $k$-algebra.
\end{notation}

In algebraic (or analytic) geometry and commutative algebra, valuations are usually treated in the {\em local setting}, and the values are always positive or non-negative. Even if it is a priori not known if a given discrete valuation $\nu$ is positive or non-negative on $R\setminus k$, it is evident how to verify this, at least if $\nu(k\setminus\{0\}) = 0$: one has only to check the values of $\nu$ on the $k$-algebra generators of $R$. For valuations {\em centered at infinity} however, in general it is non-trivial to determine if it is negative or non-positive on $R\setminus k$:
\begin{example} \label{illustration}
Let $R := \cc[x,y]$ and for every $\epsilon \in \rr$ with $0 < \epsilon < 1$, let $\nu_\epsilon$ be the valuation (with values in $\rr$) on $\cc(x,y)$ defined as follows:
\begin{align}
\nu_\epsilon(f(x,y)) := -\deg_x\left(f(x,y)|_{y = x^{5/2} + x^{-1} + \xi x^{-5/2 -\epsilon}}\right)\quad \text{for all}\ f \in \cc(x,y)\setminus\{0\},
\end{align}
where $\xi$ is a new indeterminate and $\deg_x$ is the degree in $x$. Direct computation shows that 
\begin{gather*}
\nu_\epsilon(x) = -1,\ \nu(y) = -5/2,\ \nu_\epsilon(y^2 - x^5) = -3/2,\ \nu_\epsilon(y^2 - x^5 - 2x^{-1}y) = \epsilon.
\end{gather*} 
Is $\nu_\epsilon$ negative on $\cc[x,y]$? Let $g := y^2 - x^5 - 2x^{-1}y$. The fact that $\nu_\epsilon(g) > 0$ does not seem to be of much help for the answer (especially if $\epsilon$ is very small), since $g \not\in \cc[x,y]$ and $\nu_\epsilon(xg) < 0$. However, $g$ is precisely the {\em last key form} (Definition \ref{key-defn}) of $\nu_\epsilon$ (see Example \ref{key-example}), and therefore Theorem \ref{positive-thm} implies that $\nu_\epsilon$ is {\em not} non-positive on $\cc[x,y]$, i.e.\ no matter how small $\epsilon$ is, there exists $f_\epsilon \in \cc[x,y]$ such that $\nu_\epsilon(f_\epsilon) >0$. 
\end{example}

In this article we settle the question of how to determine if a valuation centered at infinity is negative or non-positive on $R$ for the case that $R = \cc[x,y]$. At first we describe how this question arises naturally in the study of algebraic completions of affine varieties:\\ 

Recall that a {\em divisorial discrete valuation} (Definition \ref{divisorial-defn}) $\nu$ on $R$ is {\em centered at infinity} iff $\nu(f) < 0$ for some $f \in R$, or equivalently iff there is an {\em algebraic completion} $\bar X$ of $X := \spec R$ (i.e.\ $\bar X$ is a complete algebraic varieties containing $X$ as a dense open subset) and an irreducible component $C$ of $\bar X \setminus X$ such that $\nu$ is the order of vanishing along $C$. On the other hand, one way to {\em construct} algebraic completions of the affine variety $X$ is to start with a {\em degree-like function} on $R$ (the terminology is from \cite{pisis} and \cite{sub1}), i.e.\ a function $\delta: R \to \zz \cup \{-\infty\}$ which satisfy the following `degree-like' properties: 
\begin{enumerate}
\let\oldenumi\theenumi
\renewcommand{\theenumi}{P\oldenumi}
\item \label{additive-prop} $\delta(f+g) \leq \max \{\delta(f), \delta(g)\}$, and
\item \label{multiplicative-prop} $\delta(fg) \leq \delta(f) + \delta(g)$,
\end{enumerate}  
and construct the graded ring 
\begin{align}
R^\delta := \dsum_{d \geq 0} \{f \in R: \delta(f) \leq d\} \cong \sum_{d \geq 0} \{f \in R: \delta(f) \leq d\}t^d \label{k[X]-delta}
\end{align}
where $t$ is an indeterminate. It is straightforward to see that $\bar X^\delta := \proj R^\delta$ is a {\em projective completion} of $X$ provided the following conditions are satisfied: 
\begin{enumerate}
\let\oldenumi\theenumi
\renewcommand{\theenumi}{Proj-\oldenumi}
\item \label{projectively-finite} $R^\delta$ is finitely generated as a $k$-algebra, and
\item \label{projectively-positive} $\delta(f) > 0$ for all $f \in R\setminus k$. 
\end{enumerate}
A fundamental class of degree-like functions are {\em divisorial semidegrees} which are precisely the {\em negative} of divisorial discrete valuations centered at infinity - they serve as `building blocks' of an important class of degree-like functions (see \cite{pisis}, \cite{sub1}). Therefore, a natural question in this context is:
\begin{bold-question} \label{positive-question}
Given a divisorial semidegree $\delta$ on $R$, how to determine if $\delta(f) > 0$ for all $f \in R\setminus k$? Or equivalently, given a divisorial discrete valuation $\nu$ on $R$ centered at infinity, how to determine if $\nu(f)< 0$ for all $f \in R\setminus k$?
\end{bold-question}
In this article we give a complete answer to Question \ref{positive-question} for the case $k=\cc$ and $R = \cc[x,y]$ (note that the answer for the case $R = \cc[x]$ is obvious, since the only discrete valuations centered at infinity on $\cc[x]$ are those which map $x-\alpha \mapsto -1$ for some $\alpha \in \cc$). More precisely, we consider the sequence of {\em key forms} (Definition \ref{key-defn}) corresponding to semidegrees, and show that

\begin{thm} \label{positive-thm}
Let $\delta$ be a divisorial semidegree on $\cc[x,y]$ (i.e.\ $-\delta$ is a divisorial discrete valuation on $\cc[x,y]$ centered at infinity) and let $g_0, \ldots, g_{n+1}$ be the key forms of $\delta$ in $(x,y)$-coordinates. Then 
\begin{enumerate}
\item \label{non-positive-assertion} $\delta$ is non-negative on $\cc[x,y]$ iff $\delta(g_{n+1})$ is non-negative.
\item \label{positive-assertion} $\delta$ is positive on $\cc[x,y]$ iff one of the following holds:
\begin{compactenum}
\item $\delta(g_{n+1})$ is positive,
\item \label{almost-zero} $\delta(g_{n+1}) = 0$ and $g_k \not\in \cc[x,y]$ for some $k$, $0 \leq k \leq n+1$, or 
\addtocounter{enumii}{-1}
\let\oldenumii\theenumii
\renewcommand{\theenumii}{$\oldenumii'$}
\item \label{almost-zero'} $\delta(g_{n+1}) = 0$ and $g_{n+1} \not\in \cc[x,y]$.
\end{compactenum}
\end{enumerate}
Moreover, conditions \ref{almost-zero} and \ref{almost-zero'} are equivalent.
\end{thm}

\begin{rem}
The key forms of a semidegree $\delta$ on $\cc[x,y]$ are counterparts in $(x,y)$-coordinates of the {\em key polynomials} of $\nu := -\delta$ introduced in \cite{maclane-key} (and computed in local coordinates near the {\em center} of $\nu$). The basic ingredient of the proof of Theorem \ref{positive-thm} is the algebraic contratibility criterion of \cite{contractibility} which uses key forms. We note that key forms were already used in \cite{favsson-eigen}\footnote{Under the assumptions of Lemma A.12 of \cite{favsson-eigen}, the polynomials $U_j$ constructed in Section A.5.3 of \cite{favsson-eigen} are precisely the key forms of $-\nu$.} (without calling them by any special name). 
\end{rem}

\begin{rem}
The key forms of a semidegree can be computed explicitly from any of the alternative presentations of the semidegree (see e.g.\ \cite[Algorithm 3.24]{contractibility} for an algorithm to compute key forms from the {\em generic Puiseux series} (Definition \ref{dpuiseuxnition}) associated to the semidegree). Therefore Theorem \ref{positive-thm} gives an {\em effective} way to determine if a given semidegree is positive or non-negative on $\cc[x,y]$.
\end{rem}

{\em Trees of valuations centered at infinity} on $\cc[x,y]$ were considered in \cite{favsson-eigen} along with a parametrization of the tree called {\em skewness} $\alpha$. The notion of skewness has an `obvious' extension\footnote{In \cite{favsson-eigen} the skewness $\alpha$ was defined only for valuations $\nu$ centered at infinity which satisfied $\min\{\nu(x),\nu(y)\} = -1$. Here for a semidegree $\delta$, we define $\alpha(\delta)$ to be the skewness of $-\delta/d_\delta$ (where $d_\delta$ is as in \eqref{d-delta}) in the sense of \cite{favsson-eigen}. \label{alphanote}} to the case of semidegrees, and using this definition one of the assertions of \cite[Theorem A.7]{favsson-eigen} can be reformulated as the statement that the following identity holds for a certain {\em subtree} of semidegrees $\delta$ on $\cc[x,y]$:
\begin{align} 
\alpha(\delta) &= \inf\left\{\frac{\delta(f)}{d_\delta \deg(f)}: f\ \text{is a non-constant polynomial in}\ \cc[x,y]\right\},\ \text{where} \label{alphaquality} \\
d_\delta &:= \max\{\delta(x), \delta(y)\}. \label{d-delta}
\end{align}
It is observed in \cite[Page 121]{jonsson-dykovich} that in general the relation in \eqref{alphaquality} is satisfied with $\leq$, and ``it is doubtful that equality holds in general.'' Example \ref{negative-example} shows that the equality indeed does not hold in general. It is not hard to see that $\alpha(\delta)$ can be expressed in terms of $\delta(g_{n+1})$ (see \eqref{alpha-delta}), and using that expression we give a characterization of the semidegrees for which \eqref{alphaquality} holds true:

\begin{thm} \label{alpha-thm}
Let $\delta$ be a semidegree on $\cc[x,y]$ and $g_0, \ldots, g_{n+1}$ be the corresponding key forms. Then \eqref{alphaquality} holds iff one of the following assertions is true:
\begin{enumerate}
\item $\delta(g_{n+1}) \geq 0$, or
\item \label{polynomially-negative} $\delta(g_{n+1}) < 0$ and $g_k \in \cc[x,y]$ for all $k$, $0 \leq k \leq n+1$, or 
\addtocounter{enumi}{-1}
\let\oldenumi\theenumi
\renewcommand{\theenumi}{$\oldenumi'$}
\item \label{polynomially-negative'} $\delta(g_{n+1}) < 0$ and $g_{n+1} \in \cc[x,y]$.
\end{enumerate}
Moreover, the `$\inf$' in right hand side of \eqref{alphaquality} can be replaced by `$\min$' iff $g_{n+1} \in \cc[x,y]$ iff $g_k \in \cc[x,y]$ for all $k$, $0 \leq k \leq n+1$; in this case the minimum is achieved with $f = g_{n+1}$.  
\end{thm}

\begin{rem}
The right hand side of \eqref{alphaquality} can be interpreted as the {\em slope} of one of the extremal rays of the {\em cone of affine curves} in a certain compactification (namely the compactification of Proposition \ref{compact-prop}) of $\cc^2$ associated to $\delta$. 
\end{rem}

Our final result is the following corollary of the arguments in the proof of Theorem \ref{positive-thm} which answers a question of Professor Peter Russell\footnote{Prof.\ Russell's question was motivated by the correspondence established in \cite{contractibility} between normal algebraic compactifications of $\cc^2$ with one irreducible curve at infinity and algebraic curves in $\cc^2$ with one place at infinity. Since the {\em semigroup of poles} of planar curves with one place at infinity are very {\em special} (see e.g.\ \cite{abhyankar-semigroup}, \cite{sathaye-stenerson}), he asked if similarly the semigroups of values of semidegrees which determine normal algebraic compactifications of $\cc^2$ can be similarly distinguished from the semigroup of values of general semidegrees. While Example \ref{same-semi-example} shows that they can not be distinguished only by the values of the semidegree itself, Corollary \ref{semi-cor} shows that it can be done if paired with degree of polynomials.}.
 
\begin{cor} \label{semi-cor}
Let $\delta$ be a semidegree on $\cc[x,y]$. Define 
\begin{align}
S_\delta := \{(\deg(f), \delta(f)): f \in \cc[x,y] \setminus \{0\}\} \subseteq \zz^2, \label{enriques}
\end{align}
and $\scrC_\delta$ be the cone over $S_\delta$ in $\rr^2$. Then
\begin{enumerate}
\item \label{analytic-cone} $\delta$ determines an analytic compactification of $\cc^2$ iff the positive $x$-axis is not contained in the closure $\bar \scrC_\delta$ of $\scrC_\delta$ in $\rr^2$.
\item \label{algebraic-cone} $\delta$ determines an algebraic compactification of $\cc^2$ iff $\scrC_\delta$ is closed in $\rr^2$ and the positive $x$-axis is not contained in $\scrC_\delta$.
\end{enumerate}
\end{cor}

\begin{rem} \label{determining-remark}
The phrase ``$\delta$ determines an algebraic (resp.\ analytic) compactification of $\cc^2$'' means ``there exists a (necessarily unique) normal algebraic (resp.\ analytic) compactification $\bar X$ of $X := \cc^2$ such that $C_\infty := \bar X \setminus X$ is an irreducible curve and $\delta$ is the order of pole along $C_\infty$.'' In particular, $\delta$ determines an algebraic compactification of $\cc^2$ iff $\delta$ satisfies conditions \ref{projectively-finite} and \ref{projectively-positive}. 
\end{rem}

\begin{rem}
$S_\delta$ is isomorphic to the {\em global Enriques semigroup} (in the terminology of \cite{campillo-piltant-lopez-cones-surfaces}) of the compactification of $\cc^2$ from Proposition \ref{compact-prop}. Also, the assertions of Corollary \ref{semi-cor} remain true if in \eqref{enriques} $\deg$ is replaced by any other semidegree which determines an algebraic completion of $\cc^2$ (e.g.\ a weighted degree with positive weights). 
\end{rem}

\subsection*{Acknowledgements}
I am grateful to Professor Peter Russell for the question which led to this work. I would like to thank Professor Pierre Milman - the mathematics of this article was worked out while I was his postdoc at University of Toronto. The article has been written up at the Weizmann Institute as an Azrieli Fellow. I am grateful to the Azrieli Foundation for the award of an Azrieli Fellowship. 

\section{Preliminaries}
\begin{notation}
Throughout the rest of the article we write $X := \cc^2$ with polynomial coordinates $(x,y)$ and let $\bar X^{(0)} \cong \pp^2$ be the compactification of $X$ induced by the embedding $(x,y) \mapsto [1:x:y]$, so that the semidegree on $\cc[x,y]$ corresponding to the line at infinity is precisely on $\bar X^0$ is $\deg$, where $\deg$ is the usual degree in $(x,y)$-coordinates. 
\end{notation}

\subsection{Divisorial discrete valuations, semidegrees, key forms, and associated compactifications}

\begin{defn}[Divisorial discrete valuations] \label{divisorial-defn}
A discrete valuation on $\cc(x,y)$ is a map $\nu: \cc(x,y)\setminus\{0\} \to \zz$ such that for all $f,g \in \cc(x,y)\setminus \{0\}$,
\begin{compactenum}
\item $\nu(f+g) \geq \min\{\nu(f), \nu(g)\}$,
\item $\nu(fg) = \nu(f) + \nu(g)$.
\end{compactenum}
A discrete valuation $\nu$ on $\cc(x,y)$ is called {\em divisorial} iff there exists a normal algebraic surface $Y_\nu$ equipped with a birational map $\sigma: Y_\nu \to \bar X^{0}$ and a curve $C_\nu$ on $Y_\nu$ such that for all non-zero $f \in \cc[x,y]$, $\nu(f)$ is the order of vanishing of $\sigma^*(f)$ along $C_\nu$. The {\em center} of $\nu$ on $\bar X^0$ is $\sigma(C_\nu)$. $\nu$ is said to be {\em centered at infinity} (with respect to $(x,y)$-coordinates) iff the center of $\nu$ on $\bar X^0$ is contained in $\bar X^0 \setminus X$; equivalently, $\nu$ is centered at infinity iff there is a non-zero polynomial $f \in \cc[x,y]$ such that $\nu(f) < 0$. 
\end{defn}

\begin{defn}[Divisorial semidegrees] \label{semi-defn}
A {\em divisorial semidegree} on $\cc(x,y)$ is a map $\delta : \cc(x,y)\setminus\{0\} \to \zz$ such that $-\delta$ is a divisorial discrete valuation.
\end{defn}

\begin{defn}[cf.\ definition of key polynomials in {\cite[Definition 2.1]{favsson-tree}}, also see Remark \ref{key-to-key} below] \label{key-defn}
Let $\delta$ be a divisorial semidegree on $\cc[x,y]$ such that $\delta(x) > 0$. A sequence of elements $g_0, g_1, \ldots, g_{n+1} \in \cc[x,x^{-1},y]$ is called the sequence of {\em key forms} for $\delta$ if the following properties are satisfied:
\begin{compactenum}
\let\oldenumi\theenumi
\renewcommand{\theenumi}{P\oldenumi}
\addtocounter{enumi}{-1}
\item $g_0 = x$, $g_1 = y$.
\item \label{semigroup-property} Let $\omega_j := \delta(g_j)$, $0 \leq j \leq n+1$. Then 
\begin{align*}
\omega_{j+1} < \alpha_j \omega_j = \sum_{i = 0}^{j-1}\beta_{j,i}\omega_i\ \text{for}\ 1 \leq j \leq n,
\end{align*}
where 
\begin{compactenum}
\item $\alpha_j = \min\{\alpha \in \zz_{> 0}: \alpha\omega_j \in \zz \omega_0 + \cdots + \zz \omega_{j-1}\}$ for $1 \leq j \leq n$,
\item $\beta_{j,i}$'s are integers such that $0 \leq \beta_{j,i} < \alpha_i$ for $1 \leq i < j \leq n$ (in particular, $\beta_{j,0}$'s are allowed to be {\em negative}). 
\end{compactenum}

\item \label{next-property} For $1 \leq j \leq n$, there exists $\theta_j \in \cc^*$ such that 
\begin{align*}
g_{j+1} = g_j^{\alpha_j} - \theta_j g_0^{\beta_{j,0}} \cdots g_{j-1}^{\beta_{j,j-1}}.
\end{align*}

\item \label{generating-property} Let $y_1, \ldots, y_{n+1}$ be indeterminates and $\omega$ be the {\em weighted degree} on $B := \cc[x,x^{-1},y_1, \ldots, y_{n+1}]$ corresponding to weights  $\omega_0$ for $x$ and $\omega_j$ for $y_j$, $0 \leq j \leq n+1$ (i.e.\ the value of $\omega$ on a polynomial is the maximum `weight' of its monomials). Then for every polynomial $g \in \cc[x,x^{-1},y]$, 
\begin{align}
\delta(g) = \min\{\omega(G): G(x,y_1, \ldots, y_{n+1}) \in B,\ G(x,g_1, \ldots, g_{n+1}) = g\}. \label{generating-eqn}
\end{align}
\end{compactenum} 
\end{defn}

\begin{thm} \label{key-thm}
There is a unique and finite sequence of key forms for $\delta$. 
\end{thm}

\begin{rem} \label{key-to-key}
Let $\delta$ be as in Definition \ref{key-defn}. Set $u := 1/x$ and $v := y/x^k$ for some $k$ such that $\delta(y) < k\delta(x)$, and let $\tilde g_0 = u, \tilde g_1 = v, \tilde g_2, \ldots, \tilde g_{n+1} \in \cc[u,v]$ be the {\em key polynomials} of $\nu := -\delta$ in $(u,v)$-coordinates. Then the key forms of $\delta$ can be computed from $\tilde g_j$'s as follows:
\begin{align}
g_j(x,y) := \begin{cases}
				x 												& \text{for}\ j = 0,\\
				x^{k\deg_{v}(\tilde g_j)}\tilde g_j(1/x,y/x^k)  & \text{for}\ 1 \leq j \leq n+1.
			\end{cases} \label{form-from-pol}
\end{align}
Theorem \ref{key-thm} is an immediate consequence of the existence of key polynomials (see e.g.\ \cite[Theorem 2.29]{favsson-tree}).
\end{rem}

\begin{example} 
Let $(p,q)$ are integers such that $p >0$ and $\delta$ be the weighted degree on $\cc(x,y)$ corresponding to weights $p$ for $x$ and $q$ for $y$. Then the key forms of $\delta$ are $x,y$. 
\end{example}

\begin{example} \label{key-example}
Let $\epsilon := q/{2p}$ for positive integers $p, q$ such that $q < 2p$ and $\delta_\epsilon$ be the semidegree on $\cc(x,y)$ defined as follows:
\begin{align}
\delta_\epsilon(f(x,y)) := 2p\deg_x\left(f(x,y)|_{y = x^{5/2} + x^{-1} + \xi x^{-5/2 -\epsilon}}\right)\quad \text{for all}\ f \in \cc(x,y)\setminus\{0\},
\end{align}
where $\xi$ is a new indeterminate and $\deg_x$ is the degree in $x$. Note that $\delta_\epsilon = -2p\nu_\epsilon$, where $\nu_\epsilon$ is from Example \ref{illustration} (we multiplied by $2p$ to simply make the semidegree integer valued). Then the sequence of key forms of $\delta_\epsilon$ is $x,y, y^2 - x^5, y^2 - x^5 - 2x^{-1}y$. 
\end{example}

The following property of key forms can be proved in a straightforward way from their defining properties.
\begin{prop} \label{last-prop}
Let $\delta$ and $g_0, \ldots, g_{n+1}$ be as in Definition \ref{key-defn} and $d_\delta$ be as in \eqref{d-delta}. Define 
\begin{align}
m_\delta	&:= \gcd\left(\delta(g_0), \ldots, \delta(g_n)\right). \label{m-delta}		
\end{align} 
Then
\begin{align}
m_\delta \delta(g_{n+1})&\leq d_\delta^2. \label{d-m-delta-ineq}
\end{align}
Moreover, \eqref{d-m-delta-ineq} is satisfied with an equality iff $\delta = \deg$. 
\end{prop}

\begin{prop}[{\cite[Propositions 4.2 and 4.7]{sub2-1}}] \label{compact-prop}
Given a divisorial semidegree $\delta$ on $\cc[x,y]$ such that $\delta \neq \deg$ and $\delta(x) > 0$, there exists a unique compactification $\bar X$ of $\cc^2$ such that 
\begin{enumerate}
\item $\bar X$ is projective and normal.
\item $\bar X_\infty := \bar X \setminus X$ has two irreducible components $C_1,C_2$.
\item The semidegree on $\cc[x,y]$ corresponding to $C_1$ and $C_2$ are respectively $\deg$ and $\delta$. 
\end{enumerate}
Moreover, all singularities $\bar X$ are {\em rational}. Let $g_0, \ldots, g_{n+1}$ be the key forms of $\delta$. Then the {\em inverse} of the matrix of intersection numbers $(C_i, C_j)$ of $C_i$ and $C_j$, $1 \leq i,j \leq 2$, is 
\begin{align}
\scrM = \begin{pmatrix}
			1 		 & d_\delta \\
			d_\delta & m_\delta \delta(g_{n+1})
		\end{pmatrix},
\end{align} 
where $d_\delta$ and $m_\delta$ are as in respectively \eqref{d-delta} and \eqref{m-delta}.
\end{prop}

We will use the following result which is an immediate corollary of \cite[Proposition 4.2]{contractibility}.

\begin{prop} \label{polynomial-prop}
Let $\delta$, $\bar X$ and $C_1, C_2$ be as in Proposition \ref{compact-prop}. Let $g_0, \ldots, g_{n+1}$ be the key forms of $\delta$. Then the following are equivalent:
\begin{enumerate}
\item there is a (compact algebraic) curve $C$ on $\bar X$ such that $C \cap C_1 = \emptyset$.
\item $g_k$ is a polynomial for all $k$, $0 \leq k \leq n+1$.
\item $g_{n+1}$ is a polynomial.
\end{enumerate}  
\end{prop}

The following is the main result of \cite{contractibility}:

\begin{thm} \label{algebraic-thm}
Let $\delta$ be a divisorial semidegree  on $\cc[x,y]$ such that $\delta(x) > 0$ and $g_0, \ldots, g_{n+1}$ be the key forms of $\delta$. Then $\delta$ determines a normal algebraic compactification of $\cc^2$ (in the sense of Remark \ref{determining-remark}) iff $\delta(g_{n+1}) > 0$ and $g_{n+1}$ is a polynomial.
\end{thm}

\subsection{Degree-wise Puiseux series}
Note that the proof of Theorem \ref{positive-thm} does {\em not} use the material of this subsection. Proposition \ref{g_{n+1}-prop} and Corollary \ref{ratio-cor} are used in the proof of $\delta(g_{n+1}) < 0$ case of Theorem \ref{alpha-thm}. 

\begin{defn}[Degree-wise Puiseux series] \label{dpuiseuxnition}
The field of {\em degree-wise Puiseux series} in $x$ is 
$$\dpsxc := \bigcup_{p=1}^\infty \cc((x^{-1/p})) = \left\{\sum_{j \leq k} a_j x^{j/p} : k,p \in \zz,\ p \geq 1 \right\},$$
where for each integer $p \geq 1$, $\cc((x^{-1/p}))$ denotes the field of Laurent series in $x^{-1/p}$. Let $\phi = \sum_{q \leq q_0} a_q x^{q/p}$ be degree-wise Puiseux series where $p$ is the {\em polydromy order} of $\phi$, i.e.\ $p$ is the smallest positive integer such that $\phi \in \cc((x^{-1/p}))$. Then the {\em conjugates} of $\phi$ are $\phi_j := \sum_{q \leq q_0} a_q \zeta^q x^{q/p}$, $1 \leq j \leq p$, where $\zeta$ is a primitive $p$-th root of unity. The usual factorization of polynomials in terms of Puiseux series implies the following
\end{defn}

\begin{thm} \label{dpuiseux-factorization}
Let $f \in \cc[x,y]$. Then there are unique (up to conjugacy) degree-wise Puiseux series $\phi_1, \ldots, \phi_k$, a unique non-negative integer $m$ and $c \in \cc^*$ such that
$$f = cx^m \prod_{i=1}^k \prod_{\parbox{1.75cm}{\scriptsize{$\phi_{ij}$ is a con\-ju\-ga\-te of $\phi_i$}}}\mkern-27mu \left(y - \phi_{ij}(x)\right)$$
\end{thm}

\begin{prop}[{\cite[Theorem 1.2]{sub2-1}}] \label{valdeg}
Let $\delta$ be a divisorial semidegree on $\cc(x,y)$ such that $\delta(x) > 0$. Then there exists a {\em degree-wise Puiseux polynomial} (i.e.\ a degree-wise Puiseux series with finitely many terms) $\phi_\delta \in \dpsxc$ and a rational number $r_\delta < \ord_x(\phi_\delta)$ such that for every polynomial $f \in \cc[x,y]$, 
\begin{align}
\delta(f) = \delta(x)\deg_x\left( f(x,y)|_{y = \phi_\delta(x) + \xi x^{r_\delta}}\right), \label{phi-delta-defn}
\end{align}
where $\xi$ is an indeterminate. 
\end{prop}

\begin{defn} \label{generic-deg-wise-defn}
If $\phi_\delta$ and $r_\delta$ are as in Proposition \ref{valdeg}, we say that $\tilde \phi_\delta(x,\xi):= \phi_\delta(x) + \xi x^{r_\delta}$ is the {\em generic degree-wise Puiseux series} associated to $\delta$.
\end{defn}

\begin{example} 
Let $(p,q)$ are integers such that $p >0$ and $\delta$ be the weighted degree on $\cc(x,y)$ corresponding to weights $p$ for $x$ and $q$ for $y$. Then $\tilde \phi_\delta = \xi x^{q/p}$ (i.e.\ $\phi_\delta = 0$). 
\end{example}

\begin{example}
Let $\delta_\epsilon$ be the semidegree from Example \ref{key-example}. Then $\tilde \phi_\delta = x^{5/2} + x^{-1} + \xi x^{-5/2}$. 
\end{example}

The following result, which is an immediate consequence of \cite[Proposition 4.2, Assertion 2]{sub2-1}, connects degree-wise Puiseux series of a semidegree with the geometry of associated compactifications. 

\begin{prop} \label{curve-prop}
Let $\delta$, $\bar X$, $C_1$, $C_2$ be as in Proposition \ref{compact-prop} and let $\tilde \phi_\delta(x,\xi):= \phi_\delta(x) + \xi x^{r_\delta}$ be the generic degree-wise Puiseux series associated to $\delta$. Assume in addition that $\delta$ is {\em not} a weighted degree, i.e.\ $\phi_\delta(x) \neq 0$. Pick $f \in \cc[x,y]\setminus \{0\}$ and let $C_f$ be the curve on $\bar X$ which is the closure of the curve defined by $f$ on $\cc^2$. Then $C_f \cap C_1 = \emptyset$ iff the degree-wise Puiseux factorization of $f$ is of the form 
\begin{align}
\begin{aligned}
&f = \prod_{i=1}^k \prod_{\parbox{1.75cm}{\scriptsize{$\phi_{ij}$ is a con\-ju\-ga\-te of $\phi_i$}}}\mkern-27mu \left(y - \phi_{ij}(x)\right),\quad \text{where each $\phi_i$ satisfies} \\
&\phi_i(x) - \phi_\delta(x) = c_i x^{r_\delta} + \lot 
\end{aligned} \label{minimal-expansion}
\end{align}
for some $c_i \in \cc$ (where $\lot$ denotes lower order terms in $x$). 
\end{prop}

The following result gives some relations between degree-wise Puiseux series and key forms of semidegrees, and follows from standard properties of key polynomials (in particular, the first 3 assertions follow from \cite[Proposition 3.28]{contractibility} and the last assertion follows from the first; a special case of the last assertion (namely the case that $\delta(y) \leq \delta(x)$) was proved in \cite[Identity (4.6)]{sub2-1}). 

\begin{prop} \label{g_{n+1}-prop}
Let $\delta$ be a divisorial semidegree on $\cc(x,y)$ such that $\delta(x) > 0$. Let $\tilde \phi_\delta(x,\xi):= \phi_\delta(x) + \xi x^{r_\delta}$ be the generic degree-wise Puiseux series associated to $\delta$ and $g_0, \ldots, g_{n+1}$ be the key forms of $\delta$.
Then 
\begin{enumerate}
\item \label{g_{n+1}-factorization} There is a degree-wise Puiseux series $\phi$ with
$$\phi(x) - \phi_\delta(x) = c x^{r_\delta} + \lot$$
for some $c \in \cc$ (where $\lot$ denotes lower order terms in $x$) such that the degree-wise Puiseux factorization of $g_{n+1}$ is of the form 
\begin{align}
g_{n+1} &= \prod_{\parbox{1.5cm}{\scriptsize{$\phi^*$ is a con\-ju\-ga\-te of $\phi$}}}\mkern-27mu \left(y - \phi^*(x)\right).
\end{align}
\item \label{g_{n+1}-deg} Let the {\em Puiseux pairs} \cite[Definition 3.11]{contractibility} of $\phi_\delta$ be $(q_1, p_1), \ldots, (q_l,p_l)$ (if $\phi_\delta \in \cc((1/x))$, then simply set $l=0$). Set $p_0 := 1$. Then 
\begin{align*}
\deg(g_{n+1}) = \begin{cases}
					1 				 								& \text{if $\phi_\delta = 0$,} \\
					\max\{1, \deg_x(\phi_\delta)\}p_0p_1 \cdots p_l & \text{otherwise.}
				\end{cases}
\end{align*}
\item \label{d-m-delta} Write $r_\delta$ as $r_{\delta} = q_{l+1}/(p_0 \cdots p_l p_{l+1})$, where $p_{l+1}$ is the smallest integer $\geq 1$ such that $p_0 \cdots p_l p_{l+1}r_\delta$ is an integer. Let $d_\delta$ and $m_\delta$ be as in respectively \eqref{d-delta} and \eqref{m-delta}. Then 
\begin{align*}
m_\delta &= p_{l+1},\\
d_\delta &= \begin{cases}
					\max\{p_1,q_1\}	  									& \text{if $\phi_\delta = 0$,}\\
					\max\{1, \deg_x(\phi_\delta)\}p_0p_1 \cdots p_{l+1} & \text{otherwise.}
				\end{cases}
\end{align*}
\item Let the skewness $\alpha(\delta)$ of $\delta$ be defined as in footnote \ref{alphanote}. Then
\begin{align}
\begin{aligned}
\alpha(\delta) &= m_\delta\delta(g_{n+1})/d_\delta^2 		
				= 	\begin{cases}
						\frac{\min\{p_1,q_1\}}{\max\{p_1,q_1\}}  = \min\{\delta(x),\delta(y)\}/d_\delta	& \text{if $\phi_\delta = 0$,}\\
						\frac{\delta(g_{n+1})}{d_\delta \deg(g_{n+1})} 									& \text{otherwise.}
					\end{cases}
\end{aligned} \label{alpha-delta}
\end{align}
\end{enumerate}
\end{prop} 

The following lemma is a consequence of Assertion \ref{g_{n+1}-factorization} of Proposition \ref{g_{n+1}-prop} and the definition of generic degree-wise Puiseux series of a semidegree. It follows via a straightforward, but cumbersome induction on the number of {\em Puiseux pairs} of the {\em degree-wise Puiseux roots} of $f$, and we omit the proof.

\begin{lemma} \label{ratio-lemma}
Let $\delta$ be a divisorial semidegree on $\cc(x,y)$ such that $\delta(x) > 0$. Let $\tilde \phi_\delta(x,\xi):= \phi_\delta(x) + \xi x^{r_\delta}$ be the generic degree-wise Puiseux series associated to $\delta$ and $g_0, \ldots, g_{n+1}$ be the key forms of $\delta$. Then for all $f \in \cc[x,y]\setminus \cc$, 
\begin{align}
\frac{\delta(f)}{\deg(f)} \geq \frac{\delta(g_{n+1})}{\deg(g_{n+1})}. \label{delta-deg-ratio-comparison}
\end{align}
Now assume in addition that $\delta$ is {\em not} a weighted degree, i.e.\ $\phi_\delta(x) \neq 0$. Then the \eqref{delta-deg-ratio-comparison} holds with equality iff $f$ has a degree-wise Puiseux factorization as in \eqref{minimal-expansion}.
\end{lemma}

Combining Propositions \ref{polynomial-prop} and \ref{curve-prop} and Lemma \ref{ratio-lemma} yields the following

\begin{cor} \label{ratio-cor}
Consider the set-up of Proposition \ref{polynomial-prop}. assume in addition that $\delta$ is {\em not} a weighted degree. Then the Assertions 1 to 3 of Proposition \ref{polynomial-prop} are equivalent to the following statement
\begin{enumerate}
\setcounter{enumi}{3}
\item There exists $f \in \cc[x,y] \setminus \cc$ which satisfies \eqref{delta-deg-ratio-comparison} with equality.
\end{enumerate} 
\end{cor}

\section{Proofs}

\begin{proof}[Proof of Theorem \ref{positive-thm}]
W.l.o.g.\ we may (and will) assume that $\delta \neq \deg$. Let $\bar X$ be the projective compactification of $X$ from Proposition \ref{compact-prop}. In the notations of Proposition \ref{compact-prop}, the matrix of intersection numbers $(C_i,C_j)$ of $C_i$ and $C_j$, $1 \leq i,j \leq 2$, is:
\begin{align}
\scrI = \frac{1}{d_\delta^2 - m_\delta \delta(g_{n+1})}
		\begin{pmatrix}
			-m_\delta\delta\left(g_{n+1}\right)	& d_\delta \\
			d_\delta							& -1
		\end{pmatrix} \label{intersection-matrix}
\end{align} 
We consider the 3 possibilities of the sign of $\delta(g_{n+1})$ separately:

\paragraph{Case 1: $\delta(g_{n+1}) > 0$.} In this case \eqref{d-m-delta-ineq} and \eqref{intersection-matrix} imply that $(C_1,C_1) < 0$, so $C_0$ is {\em contractible} by a criterion of Grauert \cite[Theorem 14.20]{badescu}, i.e.\ there is a map $\pi: \bar X \to \bar X'$ of normal analytic surfaces such that $\pi(C_1)$ is a point and $\pi|_{\bar X \setminus C_1}$ is an isomorphism. In particular $\delta$ is the pole along the irreducible curve at infinity on the compactification $\bar X'$ of $X := \cc^2$. Consequently $\delta$ is {\em positive} on all non-constant polynomials in $\cc[x,y]$.

\paragraph{Case 2: $\delta(g_{n+1}) = 0$.} In this case $(C_1, C_1) = 0$. Note that $C_1$ is a $\qq$-Cartier divisor, since all singularities of $\bar X$ are rational. It follows that $C_1$ is a {\em nef} $\qq$-Cartier divisor. Consequently there can not be any {\em effective} curve $C$ on $\bar X$ linearly equivalent to $a_1C_1 + a_2C_2$ with $a_2 < 0$, for in that case $(C_1, C) = a_2(C_1,C_2) < 0$, which is impossible. Since the curve determined by a polynomial $f \in \cc[x,y]$ is linearly equivalent to $\deg(f)C_1 + \delta(f)C_2$, it follows that $\delta(f) \geq 0$ for all $f \in \cc[x,y]$.

\paragraph{Case 3: $\delta(g_{n+1}) < 0$.} In this case $(C_1, C_1) > 0$. It follows that $C_1$ is in the {\em interior} of the {\em cone of curves} on $\bar X$ \cite[Lemma II.4.12]{kollar-rational-curves}\footnote{Even though \cite[Lemma II.4.12]{kollar-rational-curves} is proved for only {\em non-singular} surfaces, its proof goes through for arbitrary normal surfaces using the intersection theory due to \cite{mumford-normal}.}, and consequently there are effective curves of the form $C := a_1C_1 - a_2 C_2$ with $a_1, a_2 \in \zz_{\geq 0}$ such that $a_2/a_1$ is sufficiently small. But such a curve is the closure in $\bar X$ of the curve on $\cc^2$ defined by some $f \in \cc[x,y]$ such that $\deg(f) = a_1$ and $\delta(f) = -a_2$. In particular, $\delta(f) < 0$, as required. \\

Assertion \ref{non-positive-assertion} of Theorem \ref{positive-thm} follows from Proposition \ref{polynomial-prop} and the conclusions of the above 3 cases. Assertion \ref{positive-assertion} follows from Assertion \ref{non-positive-assertion}, Proposition \ref{polynomial-prop}, and the observation from \eqref{intersection-matrix} that if $\delta(g_{n+1}) = 0$, then for every $f \in \cc[x,y]$, $\delta(f) = 0$ iff the closure in $\bar X$ of the curve on $\cc^2$ defined by $f$ does not intersect $C_1$.
\end{proof}

\begin{proof}[Proof of Theorem \ref{alpha-thm}]
W.l.o.g.\ we may (and will) assume that $\delta \neq \deg$. Let $\bar X$ be the projective compactification of $X$ from Proposition \ref{compact-prop}. We continue to use the notation of Proposition \ref{compact-prop} and divide the proof into separate cases depending on $\delta(g_{n+1})$. 

\paragraph{Case 1: $\delta(g_{n+1}) \geq 0$.} In this case Assertion \ref{non-positive-assertion} of Theorem \ref{positive-thm} implies that the {\em cone of curves} on $\bar X$ is generated (over $\rr_{\geq 0}$) by $C_1$ and $C_2$. It follows that the nef cone of $\bar X$ is
\begin{align*}
\nef(\bar X) &= \{a_1C_1 + a_2C_2: a_1, a_2 \in \rr_{\geq 0},\ (a_1C_1+a_2C_2, C_i) \geq 0,\ 1 \leq i \leq 2\}\\
			 &= \{a_1C_1 + a_2C_2:  a_1, a_2 \in \rr_{\geq 0},\ a_1 d_\delta \geq a_2 \geq a_1m_\delta \delta(g_{n+1})/d_{\delta} \}\quad \text{(using \eqref{intersection-matrix})}	 
\end{align*} 
In particular, the `lower edge' of $\nef(\bar X)$ is the half line $\{(a_1, a_2) \in \rr_{\geq 0}^2: a_2 = a_1m_\delta \delta(g_{n+1})/d_\delta\}$. Since any nef divisor is a limit of ample divisors and large multiples of ample divisors have global sections, it follows that 
\begin{align}
\frac{m_\delta}{d_\delta} \delta(g_{n+1}) = \inf\left\{\frac{\delta(f)}{\deg(f)}: f\ \text{is a non-constant polynomial in}\ \cc[x,y]\right\} \label{delta-deg-ratio}.
\end{align}
It follows from \eqref{alpha-delta} and \eqref{delta-deg-ratio} that \eqref{alphaquality} holds with equality in this case, as required. 

\paragraph{Case 2: $\delta(g_{n+1}) < 0$.} In this case it follows as in Case 3 of the proof of Theorem \ref{positive-thm} that $C_1$ is in the interior of the cone $\NE(\bar X)$ of curves on $\bar X$. \cite[Lemma II.4.12]{kollar-rational-curves} implies that $\NE(\bar X)$ has an edge of the form $\{r (C_1 - aC_2): r \geq 0\}$ for some $a > 0$, and moreover, there exists $r> 0$ such that $rC_1 - arC_2$ is linearly equivalent to some {\em irreducible} curve $C$ on $\bar X$. Pick $g \in \cc[x,y]$ such that $C \cap \cc^2$ is the zero set of $g$. Then $\deg(g) = r$ and $\delta(g) = - ar$. Since the `other' edge of $\NE(\bar X)$ is spanned by $C_2$, it follows that for all $f \in \cc[x,y]\setminus\cc$,
\begin{align}
\frac{\delta(f)}{\deg(f)} \geq \frac{\delta(g)}{\deg(g)}. \label{f-g-ratio-comparison}
\end{align}

The assertions of Theorem \ref{alpha-thm} now follow from the conclusions of the above 2 cases together with \eqref{alpha-delta}, Lemma \ref{ratio-lemma} and Corollary \ref{ratio-cor}.
\end{proof}

\begin{proof}[Proof of Corollary \ref{semi-cor}]
We continue to assume that $\delta \neq \deg$ and use the notations of the proof of Theorem \ref{alpha-thm}. Note that $\delta$ determines an analytic compactification of $\cc^2$ iff $C_1$ is {\em contractible} iff $(C_1, C_1) < 0$ (by Grauert's criterion \cite[Theorem 14.20]{badescu}) iff $\delta(g_{n+1}) > 0$ (due to \eqref{intersection-matrix}). Assertion \ref{analytic-cone} of Corollary \ref{semi-cor} then follows from identity \eqref{delta-deg-ratio} and the observation from Case 2 of the proof of Theorem \ref{alpha-thm} that if $\delta(g_{n+1}) < 0$, then the positive $x$-axis is in the interior of $C_\delta$. For Assertion \ref{algebraic-cone}, note that Lemma \ref{ratio-lemma}, Corollary \ref{ratio-cor} and identity \eqref{delta-deg-ratio} together imply that the `lower edge' of $\nef(\bar X)$ is spanned by an effective curve iff $g_{n+1}$ is a polynomial. Assertion \ref{algebraic-cone} now follows from the preceding sentence and Theorem \ref{algebraic-thm}.  
\end{proof}

\begin{example}[An example where \eqref{alphaquality} does not hold] \label{negative-example}
Let $\delta$ be the semidegree on $\cc(x,y)$ defined as follows:
\begin{align*}
\delta(f(x,y)) := \deg_x\left(f(x,y)|_{y = x^{-1} + \xi x^{-2}}\right)\quad \text{for all}\ f \in \cc(x,y)\setminus\{0\}, 
\end{align*}
where $\xi$ is an indeterminate. Then the key forms of $\delta$ are $x,y, y - x^{-1}$, and therefore \eqref{alpha-delta} implies that 
\begin{align}
\alpha(\delta) = \delta(y-x^{-1})/\deg(y-x^{-1}) = -2. \label{neg-example-alpha}
\end{align}
Now consider the surface $\bar X$ from Proposition \ref{compact-prop} and let $C$ be the closure in $\bar X$ of the curve $y = 0$ on $\cc^2$. Then in the notation of Proposition \ref{compact-prop}, $C$ is linearly equivalent to $\deg(y)C_1 + \delta(y)C_2 = C_1 - C_2$. It follows from \eqref{intersection-matrix} that $(C,C) = -1/3 < 0$, so that \cite[Lemma II.4.12]{kollar-rational-curves} implies that $C$ spans an edge of the cone of curves on $\bar X$, i.e.\ the polynomial $g$ from Case 2 of the proof of Theorem \ref{alpha-thm} is $y$. It then follows from identities \eqref{f-g-ratio-comparison} and \eqref{neg-example-alpha} that
\begin{align*}
\inf\left\{\frac{\delta(f)}{d_\delta\deg(f)}: f \in \cc[x,y]\setminus\cc \right\} = \frac{\delta(g)}{d_\delta \deg(g)} = -1 > \alpha(\delta).
\end{align*}
\end{example}

\begin{example}[The {\em semigroup of values} does not distinguish semidegrees that determine algebraic compactifications of $\cc^2$] \label{same-semi-example}
Let $\delta$ be the semidegree on $\cc(x,y)$ defined as follows:
\begin{align*}
\delta(f(x,y)) := 2\deg_x\left(f(x,y)|_{y = x^{5/2} + x^{-1} + \xi x^{-3/2}}\right)\quad \text{for all}\ f \in \cc(x,y)\setminus\{0\}, 
\end{align*}
where $\xi$ is an indeterminate. Then the key forms of $\delta$ are $x,y, y^2 - x^5, y^2 - x^5 - 2x^{-1}y$, with corresponding $\delta$-values $2, 5, 3, 1$. Since $\delta$-value of the last key polynomial is {\em positive}, it follows from the arguments of the proof of Corollary \ref{semi-cor} that $\delta$ determines an analytic compactification of $\cc^2$. But the last key form of $\delta$ is {\em not} a polynomial, so that the compactification determined by $\delta$ is {\em not} algebraic (Theorem \ref{algebraic-thm}). On the other hand, it follows from our computation of the values of $\delta$ and Corollary \ref{ratio-cor} that the semigroup of values of $\delta$ on polynomials is 
\begin{align*}
N_\delta := \{\delta(f) : f \in \cc[x,y]\} = \{2, 3, 4, \cdots \}.
\end{align*}
Now let $\delta'$ be the weighted degree on $(x,y)$-coordinates corresponding to weights $2$ for $x$ and $3$ for $y$. Then $\delta'$ determines an {\em algebraic} compactification of $\cc^2$, namely the weighted projective surface $\pp^2(1,2,3)$. But $N_\delta = N_{\delta'}$. 
\end{example}

\bibliographystyle{alpha}
\bibliography{../../utilities/bibi}

\end{document}